\documentclass[12pt]{amsart}

\calclayout

\IfFileExists{mathrsfs.sty}{\usepackage{mathrsfs}}
{\IfFileExists{euscript.sty}{\usepackage[mathscr]{euscript}}
{\let\mathscr\mathcal}}

\DeclareFontEncoding{OT2}{}{} 
 \newcommand{\textcyr}[1]{%
   {\fontencoding{OT2}\fontfamily{wncyr}\fontseries{m}\fontshape{n}%
    \selectfont #1}}
\newcommand{\Sha}{{\mbox{\textcyr{Sh}}}}

\newcommand{\Z}{\mathbb{Z}}
\newcommand{\legendre}[2]{\left(\frac{#1}{#2}\right)}
\newcommand{\Q}{\mathbb{Q}}
\newcommand{\A}{\mathbb{A}}
\newcommand{\C}{\mathbb{C}}

\newcommand{\X}{\mathscr X}
\newcommand{\Gal}{\operatorname{Gal}}
\newcommand{\rk}{\operatorname{rank\,}_{\Z}}
\newcommand{\ord}{{\operatorname{ord}}}
\newcommand{\f}{\mathfrak f_{\chi}} 
\newcommand{\fpsi}{\mathfrak f_{\psi}}

\newcommand{\End}{\operatorname{End}}
\newcommand{\Tr}{\operatorname{Tr}}
\newcommand{\Ker}{\operatorname{Ker}}

\newcommand{\kbar}{{\bar k}}
\newcommand{\isom}{\cong}
\newcommand{\ph}{\varphi}
\newcommand{\<}{\langle}
\renewcommand{\>}{\rangle}
\renewcommand{\P}{\mathbb{P}}
\newcommand{\N}{\mathfrak{F}}

\newtheorem{thm}{Theorem}[section] 
\newtheorem{prop}[thm]{Proposition}
\newtheorem{lem}[thm]{Lemma}
\newtheorem{cor}[thm]{Corollary}
\theoremstyle{definition}

\newtheorem{rem}[thm]{Remark}
\numberwithin{equation}{section}

\title[Vanishing and Non-Vanishing Dirichlet Twists]%
{Vanishing and Non-Vanishing Dirichlet Twists of $L$-Functions of Elliptic Curves}
\thanks{${}^{\dagger}$This work was supported in part by grants from NSERC and FCAR}
\date{}

\author[J.~Fearnley]{Jack Fearnley}
\address[J.~Fearnley]%
{Department of Mathematics and Statistics and CICMA\\
     Concordia University \\
     1455 de Maisonneuve  Blvd. West\\
     Montr\'eal, Quebec, H3G 1M8, CANADA}
\email{jack@mathstat.concordia.ca}
\author[H.~Kisilevsky]{Hershy Kisilevsky$^{\dagger}$}
\address[H.~Kisilevsky]%
{Department of Mathematics and Statistics and CICMA\\
     Concordia University \\
     1455 de Maisonneuve  Blvd. West\\
     Montr\'eal, Quebec, H3G 1M8, CANADA}
\email{kisilev@mathstat.concordia.ca}
\author[M.~Kuwata]{Masato Kuwata}
\address[M.~Kuwata]%
{Faculty of Economics \\
Chuo University\\
Hachioji-shi, Tokyo 192-0393,  Japan}
\email{kuwata@tamacc.chuo-u.ac.jp}

\begin{document}

\begin{abstract}
Let $L(E/\Q,s)$ be the $L$-function of an elliptic curve $E$
defined over the rational field $\Q$. We examine the vanishing
and non-vanishing of the central values $L(E,1,\chi)$ of the
twisted $L$-function as $\chi$ ranges over Dirichlet characters
of given order.
\end{abstract}

\maketitle

\section{Introduction.}

Let  $E$ be an elliptic curve defined over the field $\Q$.
Denote by
\[
L(E/\Q,s) =L(E,s)= \sum_{n\geq 1}a_n n^{-s}
\]
its $L$-function.

By the proof \cite{B-C-D-T},\cite{Ta-Wi} of the modularity of
elliptic curves over $\Q$, we know that  $L(E,s)$ has an
analytic continuation for all $s \in \C$, and satisfies the
functional equation
\[
\Lambda(E,s)= w_E \Lambda(E,2-s),
\]
where $\Lambda(E,s)=({\sqrt{N_E}/2\pi})^s\Gamma(s)L(E,s)$,
$N_E$ is the conductor of $E/\Q$, and $w_E=\pm 1$.

For a primitive Dirichlet character $\chi$ of conductor
$\f$, the {\it twist} of $L(E,s)$ by $\chi$ is
\[
L(E,s,\chi) = \sum_{n\geq 1}\chi(n)a_n n^{-s}.
\]

Then we also know that 
the $L$-function $L(E,s,\chi)$ has 
an analytic continuation and if $\f$ is coprime to $N_E$,
satisfies the functional equation
\[
\Lambda(E,s,\chi)=
w_E\chi({N_E})\tau(\chi)^2\f^{-1}\Lambda(E,2-s, \overline\chi),
\]
where
$
\Lambda(E,s,\chi)=({\f\sqrt{N_E}/2\pi})^s\Gamma(s)L(E,s,\chi)
$ and $\tau(\chi)$ is the Gauss sum 
\[
\tau(\chi)=\sum_{c=0}^{\f-1}\chi(c)\exp{(2\pi ic/\f)}.
\]

We consider the question of the vanishing or non-vanishing 
of $L(E,s,\chi)$ at $s=1$ as $\chi$ ranges over sets of 
Dirichlet characters of fixed order. For integers $n\geq 1$, let $\X(n)$ 
denote the set of primitive Dirichlet 
characters of order exactly equal to $n$, i.e.
\[
\X(n)=\{\chi \vert \chi^{n}=\chi_0 {\text{ \ and\ }} \chi^{d} \neq
\chi_0 {\text{ \  for $d<n$\ }} \}.
\]
where $\chi_0$ is the trivial character.

Given $n$ a positive integer and $X >0$, we consider  
\[
\N_E(n,X)=\N^1_E(n,X)=\#\{\chi \in \X(n) \mid \f\leq X {\text{ \ and\ }}L(E,1,\chi)=0 \},
\]
or more generally
\[
\N^r_E(n,X)=\#\{\chi \in \X(n) \mid \f\leq X {\text{ \ and\ }} \ord_{s=1}L(E,s,\chi)\geq r \}.
\]
These functions have been extensively studied in the case that
$n=2$, and there are many results and conjectures that
describe $\N^r_E(2,X)$ as $X \to \infty$. Some of these will be
reviewed in \S 2.

Given a Dirichlet character $\chi$, let $K_{\chi}$ be the cyclic
extension of $\Q$ (of conductor $\f$) which corresponds to $\chi$.  We write
$K_{\chi}=K$ when the character $\chi$ is understood.

The Birch \& Swinnerton-Dyer conjecture equates the order of vanishing of
$L(E,s)$ at $s=1$ to the $\Z$-rank of the Mordell-Weil group
$E(\Q)$. More generally (see \cite{Ro}), the order of vanishing
of $L(E,s,\chi)$ at $s=1$ is conjectured to be the rank of the
``$\chi$-component'' $E(K)^{\chi}$ of $E(K)$, where
$K=K_{\chi}$.  Here $\rk E(K)^{\chi}= \dim_{\C} \bigl(\C\otimes
E(K)\bigr)^{\chi}$ is the dimension of the $\chi$ eigenspace of
$\C\otimes_{\Z} E(K)$ as a $\Gal (\overline \Q/\Q)$-space.

The algebro-geometric version of vanishing (resp. non-vanishing)
of  $L(E,1,\chi)$ is whether the  $\chi$-component $E(K)^{\chi}$ of $E(K)$ 
has positive rank (resp. $\rk E(K)^{\chi}=0$)  
as $K_{\chi}$ ranges over the corresponding cyclic extensions of $\Q$.
This amounts to asking whether or not 
$\rk E(K_{\chi})>\rk E(F)$ for all proper subfields $F\subset K_{\chi}.$

We rely on Kato's important result \cite{Kato}
generalizing Kolyvagin's theorem \cite{Ko} which asserts that
if the $\chi$-component of $E(K_{\chi})$ has positive rank, then
$L(E,1,\chi)=0$ (see Scholl \cite{Sch}.)

Suppose that $\chi$ is a character of prime order $\ell$,
$K=K_{\chi}$ is the field corresponding to $\chi$, and let
$V=E(K)\otimes_{\Z}\Q$. Then $V$ is a representation space for
$G=\Gal(K/\Q)$ with $\dim_{\Q} V = \rk E(K)$. Since $G$ is a
cyclic group of prime order, the $\Q$-irreducible characters of
$G$ are the trivial character $\chi_0$ and an irreducible of
degree $\ell-1$ containing all the conjugates of $\chi$. Hence
if $\rk E(K)>\rk E(\Q)$, then the $\chi^j$-component of $E(K)$
has positive rank for each $j=1,\ldots,\ell-1$. It follows from Kato's
theorem (Kolyvagin \cite{Ko}, if $\ell=2$) that if $\rk E(K)^{\chi}>0$
 for a non-trivial character of prime order
$\ell$, then $L(E,1,\chi^j)=0$ for each $j=1,\ldots,\ell-1$.
In this context, it will follow from a modular symbol computation in
\S 3 that if $L(E,1,\chi)=0$ for a single character
of order $\ell$ then $L(E,1,\chi^j)=0$ for all $j=1,\ldots,\ell-1$.)

In this paper we consider the case $\ell\geq 3$. 
Our main Theorems (proved in \S 3, \S 5 and \S 6) are:

{\renewcommand{\thethm}{A}
\begin{thm}\label{thm:C}
If $L(E,1)\neq 0$, then for all but a finite number of 
{\it primes \/} $\ell$, the number of non-vanishing twists by 
Dirichlet characters of 
order $\ell$ and prime conductor satisfies
\[
\#\{\chi \in {\X}(\ell)\mid \f=p {\text{ \ prime\ } } 
<X, L(E,1,\chi) \neq 0\} \gg X/\log X. 
\]
\end{thm}
}

{\renewcommand{\thethm}{B}
\begin{thm}\label{thm:A}
If there is at least one character 
$\chi_1 \in {\X}(3)$ or $\chi_1=\chi_0,$  such
that $E(K_{\chi_1})$ is infinite, then there are infinitely many
cubic characters $\chi \in {\X}(3)$ such that 
\(
L(E,1,\chi)=0.
\)
\end{thm}
}

{\renewcommand{\thethm}{C}
\begin{thm}\label{thm:B}
Let $E/\Q$ be an elliptic curve with at least 6
rational points. Then there exist
infinitely many $\chi \in {\X}(3)$ such that the rank of the
Mordell-Weil group $\rk E(K_{\chi}) > \rk E(\Q).$ As a
consequence, there are infinitely many cubic characters $\chi
\in {\X}(3)$ such that $L(E,1,\chi)=0$.
\end{thm}
}

A random matrix model for the distribution
of zeros of $L$-functions in families was introduced by Katz and
Sarnak (\cite{Ka-Sa}) and was related to the distribution of eigenvalues of random
matrices taken from  classical groups. They proved that the model held 
in the case of $L$-functions attached to certain families of curves over finite fields.
This heuristic was applied by Conrey, Keating, Rubinstein and Snaith
(\cite{CKRS}) to give rather precise 
predictions for the frequency of vanishing of 
the central values of quadratic twists of elliptic $L$-functions
with sign $+1.$ 
In \cite{DFK1} and \cite{DFK2}, their work was adapted
to predict the frequency of vanishing of $L(E,1,\chi)$ of twists 
of the $L$-function by Dirichlet characters $\chi$ of fixed 
order greater than 2. The predictions for $\N_E(n,X)$ as $X \to \infty$ become
\begin{align*}
\N_E(n,X) &\sim b_EX^{1/2}\log^{a_E}(X) {\text{\ \   if \ \   }} \phi(n)=2\\
         &\sim \log^{a'_E}(X) {\text{\ \   if \ \  }} \phi(n)=4\\
         & {\phantom \sim}{\text{\ \   is bounded if \ \   }} \phi(n)\geq 6
\end{align*}
where $\phi$ is Euler's totient and  $b_E, a_E,a'_E\neq 0$.  
These predictions compare favourably to the numerical computations  
reported in \cite{DFK1} and \cite{DFK2}.

In \S 7 we will work out the case of a curves with rational
$3$-torsion, and for many such curves $E/\Q$ we will obtain the
strong lower bound
\[
\N_E(3,X)\gg X^{1/2}.
\]

We wish to express our appreciation to the referee for his careful 
reading of the manuscript and for his useful comments. 

\section{The Quadratic Case}
 
If $\chi$ is a {\it real \/} primitive character, i.e., if
$\chi^2=\chi_0$, then $\chi=\chi_0$ or $\chi = \legendre{\cdot}{D}$ is the
character of a quadratic field $\Q(\sqrt D)$, and $\f=D$, a
fundamental discriminant. In the latter case, $L(E,s,\chi)$ is the
$L$-function of the elliptic curve $E^D$, the twist of $E$ by
$D$. Since $\chi$ is real, the functional equation relates
$L(E,1,\chi)$ to itself and necessarily vanishes if the sign of
the functional equation  $w_{E^D}=-1$. For a primitive quadratic
character $\chi$, with $(\f, N_E)=1$, the sign of the functional
equation for $L(E,s,\chi)$ is equal to $\chi(-N_E)$ times that
of $L(E,s)$. Hence we see that $L(E,1,\chi)=0$ for at least one
half of such quadratic characters. It follows from the theorem
of Waldspurger \cite{Wa} (see also Ono-Skinner \cite{O-Sk}),
that there are an infinite number of quadratic characters $\chi$
for which $L(E,1,\chi)\neq 0$.

Gouv\^ea-Mazur \cite{Go-M} show, assuming the parity conjecture
(that $\rk E(\Q)$ of an elliptic curve, $E$, has the same parity
as $\ord_{s=1} L(E,s)$), that there are infinitely many twists
$D$ with Mordell-Weil groups of rank at least $2$. They show
(under the parity conjecture) that
\[
\N^2_E(2,X)=\#\{\chi\in {\X}(2) \mid \f<X, L(E,1,\chi)=0=L'(E,1,\chi)\}
\gg X^{1/2 - \epsilon}.
\]
Stewart-Top \cite{St-T} have removed the parity conjecture in
the Gouv\^ea-Mazur result and obtain $X^{1/7-\epsilon}$
($X^{1/6-\epsilon}$ for some special families of curves). Liem
Mai \cite{Mai} proved that the number of cubic twists of
$x^3+y^3=d$ for which the corresponding $L$-function has at
least a double zero at $s=1$ is $ \gg X^{2/3-\epsilon}$, also
assuming the parity conjecture. Goldfeld, \cite{G}, conjectured 
that for a given elliptic curve $E$ defined over $\Q$,
asymptotically one half of the quadratic twists $L(E,s,\chi)$ of
$L(E,s)$ will have a {\it simple \/} zero at $s=1$ and that
asymptotically one half will be non-vanishing.

Murty-Murty \cite{Mu-Mu} and Bump-Friedberg-Hoffstein
\cite{B-F-H} have shown that
\[
L(E,1,\chi)=0\neq L'(E,1,\chi)
\]
occurs for infinitely many quadratic characters $\chi$.

In the case of twists of $L(E,s)$ by characters of higher order,
i.e., by characters $\chi$ of order $\ell \geq 3$,  
$\chi$ assumes complex values and the functional equation relates  
$L(E,s, \chi)$ to $L(E,s,\overline\chi)$. Consequently, there is 
no longer a ``forced'' central zero due to the sign of the functional 
equation. It is an interesting question to determine the number of 
characters $\chi$ in a given set of characters $\X$
for which $L(E,1,\chi)=0$.

Rohrlich \cite{Ro2} has shown that among all Dirichlet characters $\chi$ 
with conductors supported on a finite set of primes, only finitely many
can vanish at $s=1$. Stefanicki \cite{Stef}, Akbary \cite{Ak}, and 
Murty \cite{Mu} give various nonvanishing results for twisted central 
$L$-values and of particular interest is the result of Chinta \cite{Ch}, 
which states that for sufficiently large prime $q$ 
\[
\#\{\chi \mid \f=q,L(E,1,\chi)=0\}\ll q^{7/8 +\epsilon}.
\]

\section{Non-Vanishing of Twists of Prime Order}

{\allowdisplaybreaks

Let $\ell>2$ denote an {\it odd prime \/} number and suppose that
$E$ is an elliptic curve defined over the rational field $\Q$.
Let $\chi \in \X(\ell)$ be a Dirichlet character.
We will demonstrate a congruence between the algebraic part
of $L(E,1)$ and the algebraic part $L(E,1,\chi)$. In the case 
that $L(E,1)\neq 0$ this will allow us to prove that for all but 
a finite number of primes $\ell$, there are infinitely many characters 
$\chi \in \X(\ell)$ such that $L(E,1,\chi)\neq 0$.

Let $f$ be a weight two modular form of level $N$. 
We recall the properties of modular symbols
following  Mazur, Tate and Teitelbaum (\cite{M-T-T}).

For $\alpha$ and $\beta$ in the upper half plane, define a modular 
symbol $\{\alpha ,\beta \}$ as a linear 
functional on cuspforms $f\in S_2(N)(=S_2(\Gamma_0(N))$ by
\[
\{\alpha ,\beta \}f=2\pi i \cdot\int_{\alpha }^{\beta }f(z)dz.
\]

For a fixed  cuspform $f\in S_2(N)$, we will write  $\{\alpha ,\beta \}$ for
$\{\alpha ,\beta \}f$.
The properties of the modular symbol which are important for our purposes
are summarized below:

\begin{prop}[$L$-function relation]
The $L$-series of a cuspform at its critical point can be expressed
as a modular symbol
\[
L(f,1)=\left\{ \infty ,0\right\} .
\]
\end{prop}

\begin{prop}[Birch's Theorem] 
The value of an $L$-series twisted by a Dirichlet character can be expressed
as a weighted sum of modular symbols
\[
L(f,1,\chi )=\frac{\tau (\chi )}{\f}\sum_{a \bmod \f}\overline{\chi}(a)\left\{
\infty ,\frac{a}{\f}\right\}
\]
where $\tau (\chi )$ is the Gauss sum.
\end{prop}

\begin{prop}[Hecke action]
For an eigenform $f$ of the Hecke operator $T_{p}$ with eigen\-value $a_{p}$ we have
an action on the modular symbol as follows
\[
a_{p}\left\{ \infty ,\frac{a}{\f}\right\} =\left\{ \infty ,\frac{a}{\f}
\right\} ^{T_{p}}=\sum_{u=0}^{p-1}\left\{ \infty ,\frac{a-u\f}{p\f}\right\}
+\delta (p)\left\{ \infty ,\frac{ap}{\f}\right\}
\]
where $\delta (p)=0$ if $p|N$ and  $\delta (p)=1$ otherwise.
\end{prop}

\begin{prop}[Integrality]
There are non-zero complex numbers
$\Omega ^{\pm }$ depending only upon $f$ such that
\[
\Lambda ^{\pm }(a,\f):=\left( \left\{ \infty ,\frac{a}{\f}\right\} \pm \left\{
\infty ,\frac{-a}{\f}\right\} \right) /\Omega ^{\pm }\text{ are integers.}
\]

\end{prop}

In the case that $f$ is the cuspform associated to an elliptic curve, then
the numbers $\Omega ^{\pm }$ are rational multiples of the periods 
of the elliptic curve. As in \cite{M-T-T}, choose such an $\Omega ^{\pm }$ (up to sign) so 
that the set of integers $\Lambda ^{\pm }(a,\f)$ have greatest common divisor 
equal to one.

Note that for $\gamma \in \Gamma_0(N),$ and $f \in S_2(\Gamma_0(N)),$ that
$\{\gamma(\alpha) ,\gamma(\beta) \}f=\{\alpha ,\beta \}f.$ It follows
that $\Lambda ^{\pm }(a,\f)$ depends only on the residue class of $a \mod \f.$

Since we consider characters $\chi$ of prime order $\ell>2$, they will be  
even characters (i.e.~$\chi(-1)=1) $ and we
then take the positive sign. In what follows we write $\Lambda $ 
for $\Lambda ^{+}$ and $\Omega$ for $\Omega^{+} $.

Following \cite{M-T-T}
define the algebraic part of $L(f,1,\chi)$ to be
 \begin{align*}
L^{\text{alg}}(f,1,\chi)&=\frac{2\f L(f,1,\chi)}{\Omega\tau(\chi)}\\
                        &= \sum_{a\bmod \f}\overline{\chi}(a)\Lambda(a,\f)
\end{align*}
 where $\Omega=\Omega^+$, chosen as above
is independent of $\chi$ and $\Lambda(a,\f)\in\Z$.

Let $\mathfrak l$ be a prime dividing $\ell$ in the cyclotomic field $\Q(\zeta_{\ell})$
of $\ell$-th roots of unity and let $\chi \in \X(\ell)$ be a Dirichlet 
character with conductor $\f$. Then
\begin{align*}
 \chi (a) &\equiv 1 \bmod \mathfrak{l} \text{  when  } (a,\f)=1 \\
 \chi (a) &= 0 \text{  when  } (a,\f)\neq 1.
\end{align*}

So
\[
\sum_{a \bmod \f}\overline{\chi}(a)\Lambda (a,\f)\equiv 
\sum\limits_{\substack{a \bmod \f \\
(a,\f)=1  }} \Lambda (a,\f) \bmod \mathfrak{l}.
\]

Fix a cuspform $f\in S_2(N)$, let
\[
\Sigma_{m}(t):=\sum\limits_{\substack{ a \bmod t  \\ (a,m)=1  }} \Lambda (a,t).
\]
For a character of order $\ell$ and conductor $\f$ we have
\[
L^{\text{alg}}(f,1,\chi )\equiv \Sigma_{\f}(\f) \bmod \mathfrak{l}.
\]

\begin{thm}\label{thm:alg_part_of_L}
Let $f\in S_2(N)$ be a simultaneous eigenform for all the
Hecke operators. Let $\chi $ be Dirichlet character of order dividing $\ell$ and 
conductor $\f$, and let $\psi \in \X(\ell) $  
and prime conductor $\fpsi=p$ with $(\f,p)=1$. 
Let $\delta (t)=1$ if $(t,N)=1$ and zero otherwise. Then 
\[
L^{\text{alg}}(f,1,\chi \psi )\equiv (a_{p}-\delta (p)-1)
L^{\text{alg}}(f,1,\chi ) \bmod \mathfrak l.
\]
If $\varphi $ is the Dirichlet character of order $\ell$ and conductor $\ell^{2}$
prime to $\f$ we have
\[
L^{\text{alg}}(f,1,\chi \varphi )\equiv (a_{\ell}-1)(a_{\ell}-\delta
(\ell))L^{\text{alg}}(f,1,\chi ) \bmod \mathfrak l.
\]
\end{thm}

Since any character $\psi$ of order $\ell$ and conductor $\fpsi$ 
can be factored as a product of characters of order $\ell$ either 
with prime conductors, or with conductor $\ell^2$, we can 
iterate the above result to obtain:

\begin{cor} 
For $f\in S_2(N)$ as above, $\chi$ a character of order dividing $\ell$, and $\psi\in\X(\ell)$, if $\fpsi$ is not 
divisible by $\ell$ then we have 
\[
L^{\text{alg}}(f,1,\chi\psi)\equiv 
  L^{\text{alg}}(f,1,\chi)\prod_{p\mid \fpsi}(a_p-\delta(p)-1) \bmod {\mathfrak l}.
\]
or if $\ell \mid \fpsi$
\[
L^{\text{alg}}(f,1,\chi\psi)\equiv 
  L^{\text{alg}}(f,1,\chi)(a_{\ell}-1)(a_{\ell}-\delta(\ell))
\prod\limits_{\substack{p\mid \fpsi \\
p\neq \ell}}(a_p-\delta(p)-1) \bmod {\mathfrak l}.
\]
\end{cor}

\begin{proof}[Proof of Theorem~\ref{thm:alg_part_of_L}]

We consider the sums $\Sigma_{m}(m),\ \Sigma_{pm}(pm)$, and $\Sigma_{p^{2}m}(p^{2}m)$. Assume that $(m,p)=1$.
\begin{align*}
 \Sigma_{m}(m)\mid T_{p} &=a_{p}\Sigma_{m}(m)=\sum\limits_{\substack{ a \bmod m  \\(a,m)=1  }} 
\left[ \sum_{u=0}^{p-1}\Lambda (a-um,pm)+\delta(p)\Lambda (ap,m)\right] \\
&=\sum\limits_{\substack{ a \bmod m  \\ (a,m)=1  }}\sum_{u=0}^{p-1}\Lambda
(a-um,pm)+\delta (p)\sum\limits_{\substack{ a \bmod m  \\ (a,m)=1  }} \Lambda
(ap,m) \\
&=\Sigma_{m}(pm)+\delta (p)\Sigma_{m}(m).
\end{align*}

Now
\begin{align*}
\Sigma_{m}(pm) &=\sum\limits_{\substack{ a \bmod pm  \\ (a,m)=1  }} \Lambda (a,pm) \\
&=\sum\limits_{\substack{ a \bmod pm  \\ (a,pm)=1  }} \Lambda (a,pm)+
\sum\limits_{\substack{ a \bmod pm  \\ (a,pm)=p  }} \Lambda (a,pm) \\
&=\Sigma_{pm}(pm)+\sum\limits_{\substack{ b \bmod m  \\ (b,m)=1  }} \Lambda (bp,pm) \\
&=\Sigma_{pm}(pm)+\Sigma_{m}(m).
\end{align*}

So
\begin{align*}
a_{p}\Sigma_{m}(m) &=\Sigma_{pm}(pm)+\Sigma_{m}(m)+\delta (p)\Sigma_{m}(m) \\
\Sigma_{pm}(pm) &=(a_{p}-1-\delta (p))\Sigma_{m}(m).
\end{align*}
Taking $m=\f$, and noting that $\fpsi=p$ we have 
\begin{align*}
L^{\text{alg}}(f,1,\chi \psi )&\equiv \Sigma_{pm}(pm)\bmod \mathfrak l\\
L^{\text{alg}}(f,1,\chi )&\equiv \Sigma_{m}(m) \bmod \mathfrak l.
\end{align*}
Therefore the first statement of Theorem~\ref{thm:alg_part_of_L} follows:
\[
L^{\text{alg}}(f,1,\chi \psi )\equiv (a_{p}-\delta (p)-1)
L^{\text{alg}}(f,1,\chi ) \bmod \mathfrak l.
\]

To treat $\Sigma_{p^{2}m}(p^{2}m)$ we apply $T_{p}$ a second time.
\begin{align*}
\left( \Sigma_{m}(m)\mid T_{p}\right) \mid  T_{p}
&=a_{p}^{2}\Sigma_{m}(m)=\sum\limits_{\substack{ a \bmod m  \\ (a,m)=1  }} \left[
\sum_{u=0}^{p-1}\Lambda (a-um,pm)^{T_{p}}+\delta (p)\Lambda
(ap,m)^{T_{p}}\right] \\
&=\sum\limits_{\substack{ a \bmod m  \\ (a,m)=1  }} \sum_{v=0}^{p-1}\sum_{u=0}^{p-1}
\Lambda (a-um-vmp,p^{2}m) \\
&\ \ \ +\delta (p)\sum\limits_{\substack{ a \bmod m  \\ (a,m)=1  }} 
\sum_{v=0}^{p-1} \Lambda((a-um)p,pm) \\
&\ \ \ +\delta (p)\sum\limits_{\substack{ a \bmod m  \\ (a,m)=1  }} 
\sum_{v=0}^{p-1}\Lambda (ap-vm,pm)+\delta (p)^{2}\sum\limits_{\substack{ a \bmod m
\\ (a,m)=1  }} \Lambda (ap^{2},m) \\
&=\Sigma_{m}(p^{2}m)+\delta (p)p\Sigma_{m}(m)+\delta
(p)\Sigma_{m}(pm)+\delta (p)\Sigma_{m}(m).
\end{align*}
Now
\[
\Sigma_{m}(pm)=\Sigma_{pm}(pm)+\Sigma_{m}(m)
\]
and
\begin{align*}
\Sigma_{m}(p^{2}m) &=\sum\limits_{\substack{ a \bmod p^{2}m  \\ (a,m)=1  }} \Lambda
(a,p^{2}m) \\
&=\sum\limits_{\substack{ a \bmod p^{2}m  \\ (a,mp^{2})=1  }} \Lambda (a,p^{2}m)+
\sum\limits_{\substack{ a \bmod p^{2}m  \\ (a,mp^{2})=p  }} \Lambda (a,p^{2}m)+
\sum\limits_{\substack{ a \bmod p^{2}m  \\ (a,mp^{2})=p^{2}  }} \Lambda (a,p^{2}m) \\
&=\Sigma_{p^{2}m}(p^{2}m)+\sum\limits_{\substack{ b \bmod pm  \\ (b,mp)=1  }} \Lambda
(bp,p^{2}m)+\sum\limits_{\substack{ c \bmod m  \\ (c,m)=1  }} \Lambda (cp^{2},p^{2}m)
\\
&=\Sigma_{p^{2}m}(p^{2}m)+\Sigma_{pm}(pm)+\Sigma_{m}(m).
\end{align*}
So
\begin{align*}
a_{p}^{2}\Sigma_{m}(m) &=\Sigma_{p^{2}m}(p^{2}m)+\Sigma_{pm}(pm)+\Sigma_{m}(m)+\delta
(p)p\Sigma_{m}(m)+\delta (p)\Sigma_{m}(m) \\
&\ \ \ +\delta (p)(\Sigma_{pm}(pm)+\Sigma_{m}(m)) \\
&=\Sigma_{p^{2}m}(p^{2}m)+\Sigma_{pm}(pm)(1+\delta (p))+\Sigma_{m}(m)(1+\delta
(p)p+2\delta (p)) \\
&=\Sigma_{p^{2}m}(p^{2}m) \\
&\ \ \ +(a_{p}-1-\delta (p))(1+\delta
(p))\Sigma_{m}(m)+\Sigma_{m}(m)(1+\delta (p)p+2\delta (p)) \\
&=\Sigma_{p^{2}m}(p^{2}m)+\Sigma_{m}(m)(a_{p}+\delta (p)a_{p}-\delta
(p)+\delta (p)p).
\end{align*}
Simplifying
\begin{align*}
\Sigma_{p^{2}m}(p^{2}m) &=\left[ a_{p}^{2}-a_{p}-\delta
(p)a_{p}-\delta (p)p+\delta (p)\right] \Sigma_{m}(m) \\
&=\left[ (a_{p}-1)(a_{p}-\delta (p))-\delta (p)p\right] \Sigma_{m}(m).
\end{align*}
Taking $p=\ell$ we see that the second statement of Theorem~\ref{thm:alg_part_of_L} now follows as above.
If $\varphi $ is the Dirichlet character of order $\ell$ and conductor $\ell^{2}=p^2$
prime to $\f$ we have
\begin{align*}
L^{\text{alg}}(f,1,\chi \varphi )&\equiv \Sigma_{p^{2}m}(p^{2}m)  \bmod \mathfrak l\\
            &\equiv \left[ (a_{\ell}-1)(a_{\ell}-\delta (\ell))-\delta (\ell)\ell\right] \Sigma_{m}(m)\bmod \mathfrak l\\
&\equiv (a_{\ell}-1)(a_{\ell}-\delta(\ell))L^{\text{alg}}(f,1,\chi ) \bmod \mathfrak l.
\end{align*}
\end{proof}

\begin{thm}\label{thm:non_vanishing_of_L^alg}
Let $E/\Q$ be an elliptic curve and let $\ell$ be a 
prime. Suppose that $L^{\text{alg}}(E,1) \not\equiv 0 \bmod  \ell$, 
then there is a set of primes, $S$, of 
positive density such that $L(E,1,\chi) \neq 0$
for any characters $\chi$ of order $\ell$ with conductor $\f$
supported on $S$.
\end{thm}

\begin{proof}
The  value of the twist  $L^{\text{alg}}(E,1,\chi)$
can only vanish $\bmod {\,\mathfrak l}$ if either
$L^{\text{alg}}(E,1) \equiv 0 \bmod {\mathfrak l}$ or if one of
the factors $(a_p-\delta(p)-1)$ or
$(a_{\ell}-1)(a_{\ell}-\delta(\ell))$ vanishes $\bmod {\,\mathfrak
l}$.

By Serre's theorem, (\cite{S}) we can find a
{\it positive density \/} of primes $p\equiv 1 \bmod \ell$ for
which the factors $(a_p-\delta(p)-1) \not\equiv 0 \bmod \ell$.
\end{proof}

\begin{rem} The following examples show that the condition $L^{\text{alg}}(E,1) \not\equiv 0 \bmod  \ell$
is necessary for the argument proving Theorem 3.7. Consider the elliptic curve $ E= 2534$f$1$ in Cremona's tables \cite{Cr}.
Then  $L^{\text{alg}}(E,1)=18$ but $L^{\text{alg}}(E,1,\chi)=0$ for $\chi$  the character of order $3$ and conductor $\f=37.$
We note that for this curve $a_{37} =4$ so $a_{37}-\delta(37)-1=2 \not\equiv 0 \bmod 3$. Similar examples are afforded by
the elliptic curves 2718d1 twisted by the character of order 3 and conductor 523 and 4229a1 twisted by the character of order 3 and 
conductor 43. We would like to thank the referee for pointing out the need for such examples.
\end{rem}

\begin{thm}
Let $E/\Q$ be an elliptic curve such that
$L(E,1) \neq 0$. Then for all but a finite number of primes $\ell$,
there is a set of primes  $S_{\ell}$ of positive density such that
$L(E,1,\chi) \neq 0$ for all characters $\chi$ of order $\ell$ with 
conductor $\f$ supported on $S_{\ell}$.
\end{thm}
} 

The statement of Theorem~\ref{thm:C} follows immediately.

\section{Vanishing Twists and Rational Points of an Auxiliary Variety}

Kato's result \cite{Kato} generalizing Kolyvagin's theorem shows
that if the $\chi$-component of $E(K_{\chi})$ has positive rank, then
$L(E,1,\chi)=0$. 
Thus, the algebro-geometric version of our question is whether the 
$\chi$-component $E(K_{\chi})^{\chi}$ of $E(K_{\chi})$ has positive 
rank or not, as $K_{\chi}$ ranges over the corresponding cyclic 
extensions of $\Q$. 
In order to find points on $E$ defined over some cyclic 
extension $K_{\chi}$, we will
define an auxiliary variety of higher dimension whose $\Q$-rational points
correspond to points on $E$ defined over some cyclic extensions. 
As we are interested in cubic extensions, we construct this variety starting from $E^{3}=E\times E\times E$.

Let $k$ be a number field, and let $\bar k$ be its algebraic closure, which we fix once and for all.  We denote by $G_{k}=\Gal(\bar k/k)$, the Galois group of $\bar k/k$.  Throughout this section, a point means a geometric point, i.e.,a $\bar k $-valued point. 

The symmetric group $\mathfrak{S}_{3}$ acts on $E^{3}$ in an obvious way.  Its alternating subgroup $\mathfrak{A}_{3}$ is a cyclic group of order~$3$ generated by the automorphism $\rho$ given by 
\[
\begin{array}{rcccc}
\rho& : &E\times E  \times E &\longrightarrow & E\times E \times E \\
&& (P,Q,R)&\longmapsto & (Q,R,P).
\end{array}
\]
Let $X=E^{3}/\mathfrak{A}_{3}$ be the quotient variety  (cf.\ \cite[Ch.~II, \S7 and
Ch.~III, \S12)]{mumford:abelian-varieties}).   Since the action of $\mathfrak{A}_{3}$ commutes with the Galois action of $E^{3}$, $X$ is a variety defined over $k$.  We denote by $[P,Q,R]$ the class of $(P,Q,R)$ in $X$.  

Let $D$ be the diagonal subgroup $\{(P,P,P)\mid P\in E\}$  in $E^{3}$.  We define the complement $\widehat D$ of $D$ in $E^{3}$ by 
\[
\widehat D = \{(P,Q,R)\mid P+Q+R = O\}.
\]
$\widehat D$ is a subgroup of $E^{3}$ invariant under the action of $\rho$.  We have $(D \times {\widehat D})/\mathfrak{A}_{3}=D\times ({\widehat D}/\mathfrak{A}_{3})\isom E\times ({\widehat D}/\mathfrak{A}_{3})$.  Moreover, we have a degree $3$ isogeny $\ph:E^{3}\to D\times \widehat D$ given by
\begin{alignat*}4
&\ph\ :&&\qquad E^{3} & \ &\longrightarrow \ & D&\times {\widehat D}
\\
&\null&&(P,Q,R) &&\longmapsto &
\textstyle \bigl(
(S,S,S) &,
\textstyle (3P - S,3Q - S,3R - S)\bigr),
\end{alignat*}
where $S=P+Q+R$.  Its dual isogeny $\ph': D\times \widehat D\to E^{3}$ is given 
\begin{alignat*}4
&\ph'\ :& D&\times {\widehat D} & \ &\longrightarrow &\ &\qquad\qquad E^{3} \\
& &
\bigl((S,S,S)&,(P,Q,R)\bigr)
&&\longmapsto &&
(P+S,Q+S,R+S). 
\end{alignat*}
Since $\ph$ and $\ph'$ commute with the $\mathfrak{A}_{3}$-action, we take the
quotients by the $\mathfrak{A}_{3}$-action and obtain two maps
\[
\bar\ph : X \to E\times ({\widehat D}/\mathfrak{A}_{3}), \qquad
\bar\ph' :E\times ({\widehat D}/\mathfrak{A}_{3}) \to X
\]
such that $\bar\ph'\circ\bar\ph$ is the map $\overline{[3]}$
induced from the multiplication-by-$3$ map $[3]$ of $E^{3}$.
The quotient  $\widehat D/\mathfrak{A}_{3}$ is a surface, which 
we denote  by $\overline{S}_{E}$.

\begin{lem}\label{prop:rat_pts_on_SE}
A point $[P,Q,R]$ in $\overline{S}_{E}$ is a $k$-rational point if and only if one of the following is satisfied:
\begin{enumerate}
\item $P$, $Q$  and $R$ are all defined over $k$, and $P+Q+R=O$,
 or 
\item $P$, $Q$  and $R$ are defined over a certain cyclic cubic extension $K/k$, and for a suitable generator $\sigma\in \Gal(K/k)$, we have $(P^{\sigma},Q^{\sigma},R^{\sigma})=(Q,R,P)$, together with the relation $P+P^{\sigma}+P^{\sigma^{2}}=O$.
\end{enumerate} 
\end{lem}

\begin{proof}
It is clear that if (1) or (2) is satisfied, then $[P,Q,R]$ is stable under the Galois action, and thus $[P,Q,R]\in \overline{S}_{E}(k)$.

Conversely, suppose $[P,Q,R]\in \overline{S}_{E}(k)$.  This is equivalent to say that, for any $\sigma\in G_{k}$, the conjugate $(P^{\sigma}, Q^{\sigma}, R^{\sigma})$ coincides with $(P,Q,R)$, $\rho((P,Q,R))=(Q,R,P)$, or $\rho^{2}((P,Q,R))=(R,P,Q)$.   

First, suppose that the triples $(P,Q,R)$, $(Q,R,P)$ and $(R,P,Q)$ are mutually distinct.  Then, we can define a map $\psi:G_{k}\to \<\rho\>$ by stipulating that 
$(P^{\sigma}, Q^{\sigma}, R^{\sigma}) = \psi(\sigma)((P,Q,R))$.   
Since the automorphism $\rho$ is defined over $k$, the map $\psi$ is a homomorphism.  Let $K$ be the Galois extension of $k$ corresponding to $\Ker\psi$ via Galois theory.  Then $\Gal(K/k)$ is isomorphic to a subgroup of $ \<\rho\>$, that is $\Gal(K/k)\isom  \<\rho\>$, or $\Gal(K/k)=\{\text{id}\}$.  
If $\Gal(K/k)=\{\text{id}\}$, then $K=k$, and $P$, $Q$, $R$ are all defined over $k$.  This is the case (1).  If $\Gal(K/k)\isom  \<\rho\>$, then $K$ is a cyclic cubic extension.  Let $\sigma\in \Gal(K/k)$ be the element that maps to $\rho$ by $\psi$.  Then we have $(P^{\sigma},Q^{\sigma},R^{\sigma})=(Q,R,P)$.  This is the case (2). 

Next, if $(P,Q,R)$, $(Q,R,P)$ and $(R,P,Q)$ are not mutually distinct, then we must have $P=Q=R$.  In this case, for any $\sigma\in G_{k}$, the conjugate $(P^{\sigma}, P^{\sigma}, P^{\sigma})$ must equal $(P,P,P)$.  Thus, this case is included in the case~(1).
\end{proof}

In order to give a concrete description of $\widehat D/\mathfrak{A}_{3}$, 
we fix a Weierstrass model of $E$ and consider it as a curve in 
$\P^{2}$.  Namely, suppose that $E$ is given by the equation
\begin{equation*}
E: y^{2}z+ a_{1}xyz + a_{3}yz^{2} 
= x^{3} + a_{2}x^{2}z + a_{4}xz^{2} + a_{6}z^{3}.
\end{equation*}
Then, as is well known, three points $P$, $Q$ and $R$ satisfy
$P+Q+R=O$ if and only if $P$, $Q$ and $R$ are collinear. Let
$(\P^{2})^{*}$ be the dual space of $\P^{2}$, namely, the space of all
the lines in $\P^{2}$. For a point $(P, Q, R)\in\widehat D$ we denote by
$\ell_{PQR}$ the line passing through $P$, $Q$ and $R$. If $P=Q=R$, we
understand that $\ell_{PPP}$ is the tangent line to $E$ passing through $P$.
Consider the map
\[
\begin{array}{rccc}
\pi_{0}:  & \widehat D &\longrightarrow &(\P^{2})^{*} \\
 & (P,Q,R) &\longmapsto &\ell_{PQR}.
\end{array}
\]
It is clear
that $\pi_{0}$ is a surjection and is invariant under the $\mathfrak{S}_{3}$-action.
Thus we obtain an isomorphism 
$\widehat D/\mathfrak{S}_{3}\overset{\simeq}{\to} (\P^{2})^{*}$, which sends a class
$[P,Q,R]$ of $\widehat D/\mathfrak{S}_{3}$ to the line $\ell_{PQR}$.
Now, $\pi_{0}$ induces the map
\[
\pi_{1}: \overline{S}_{E}=\widehat D/\mathfrak{A}_{3} 
\longrightarrow \widehat D/\mathfrak{S}_{3}\simeq(\P^{2})^{*}.
\]
$\pi_{1}$ is a covering map of degree~$2=[\mathfrak{S}_{3}:\mathfrak{A}_{3}]$.
It is easy to
see that $\pi_{1}^{-1}(\ell_{PQR})$ consists of two classes,
$[P,Q,R]$ and $[P,R,Q]$. In $\overline{S}_{E}$ the classes $[P,Q,R]$ and
$[P,R,Q]$ coincide if and only
if at least two of three points coincide. In other words $[P,Q,R]=[P,R,Q]$ if and only
if $\ell_{PQR}$ is a tangent line to the curve $E$. This implies that
the double covering $\pi_{1}$ ramifies along
the dual curve $E^{*}=\{L\in(\P^{2})^{*}\mid \text{$L$ is tangent to $E$}\}$. $E^{*}$ is an irreducible curve of degree~$6$, 
and it has nine cusps corresponding to the tangent lines at nine
inflection points of $E$.
The surface $\overline{S}_{E}$ thus has nine singular points of type $A_{2}$. 
Let $S_{E}$ be the minimal desingularization of $\overline{S}_{E}$ obtained by blowing up twice at each singular points. 
Summing all up, we have

\begin{prop}
The quotient surface $\overline{S}_{E}=\widehat D/\mathfrak{A}_{3}$ may be
regarded as a double cover of the dual projective plane $(\P^{2})^{*}$ ramifying
along the dual curve $E^{*}$ of $E$, which is an irreducible curve of degree~$6$. 
As a consequence the minimal desingularization $S_{E}$ of $\overline{S}_{E}$ is 
a $K3$ surface. \qed
\end{prop}

\begin{rem}
If the quotient of an abelian surface $A$ by a finite group $G$ is birational 
to a $K3$ surface, its minimal 
desingularization is called a generalized Kummer surface.  $S_{E}$ is thus a generalized Kummer surface.
For more about Kummer surfaces see Katsura \cite{katsura} and Bertin \cite{bertin}. 
\end{rem}

Write the equation of a generic line $\ell$ in $\P^{2}$ in the form $y=tx+u$,
using parameters $t$ and $u$.
The function field of $(\P^{2})^{*}$ is then given by $k(t,u)$, and 
the function field of $\widehat D$ can be regarded as the splitting field of
the cubic equation in $x$ obtained by substituting $y=tx+u$ in the affine
Weierstrass equation
\begin{equation*}
y^{2}+ a_{1}xy + a_{3}y
= x^{3} + a_{2}x^{2} + a_{4}x  + a_{6}.
\end{equation*}
The function field of $\overline{S}_{E}=\widehat D/\mathfrak{A}_{3}$ is then the extension of $k(t,u)$ obtained by adding the square root of the discriminant $\Delta(u,t)$ of this cubic equation with respect to~$x$.  This implies that the surface $S_{E}$ is a minimal model of the surface defined by the equation
\begin{equation}\label{S_E_affine}
\delta^{2}=\Delta(u,t).
\end{equation}
For simplicity we write the explicit result only in the 
case where $a_{1}=a_{2}=a_{3}=0$, $a_{4}=A$, and $a_{6}=B$.

\begin{prop}
Let $y^{2}=x^{3}+Ax+B$ be a Weierstrass equation for $E$.  Then, the generalized Kummer surface $S_{E}$ is birational to the affine surface in $\A^{3}=\{(t,u,\delta)\}$ defined
by the equation
\begin{multline}\label{quartic}
	\delta^{2} = -27 u^{4} - 4 t^{3} u^{3} - (30 A t^{2} - 54 B) u^{2} 
	- 4 t ( A t^{4} - 9 t^{2} - 6 A^{2} ) u \\
	+ 4 B t^{6} + A^{2} t^{4} - 18 A B t^{2} - ( 4 A^{3} + 27 B^{2} ).\qed
\end{multline}
\end{prop}

\begin{rem}
The surface $S_{E}$ possesses two obvious involutions, 
$[P,Q,R]\mapsto [Q,P,R]$ and $[P,Q,R]\mapsto [-P,-Q,-R]$. In terms of the 
equation~\eqref{quartic}, the former corresponds to 
$(t,u,\delta)\mapsto(t,u,-\delta)$, while the latter corresponds to 
$(t,u,\delta)\mapsto(-t,-u,\delta)$.
\end{rem}

Consider the map $\nu:E\to \overline{S}_{E}$ given by $P\mapsto
[P,-P,O]$.  This is an injection, and we have an embedding
$\tilde\nu:E\to S_{E}$.  Let $D_{\tilde\nu(E)}$ be the divisor associated with
the image of $\tilde\nu$.  Then the complete linear system $|D_{\tilde\nu(E)}|$
determines a pencil of curves of genus~$1$.  Let $\bar\pi:\overline{S}_{E}\to \P^{1}$
be the map  associated with the projection $(t,u,\delta)\mapsto t$.
The fiber at $t=\infty$ corresponds exactly the image of the embedding 
$P\to [P,-P,O]$, and thus the fibration $\pi$ coincides with the pencil above.
Let  $\pi:S_{E}\to \P^{1}$ be the elliptic fibration obtained in this way.  

Let $C_{t}$ be the generic fiber of $\pi$, that is, the curve of genus~$1$ 
defined over the function field $k(t)$ given by the equation~\eqref{quartic}.

The coefficient of $u^{4}$ in the right-hand side of \eqref{quartic} 
is constant, namely, $-27$.  Thus, the curve $C_{t}$
has two points at infinity defined over $k(\sqrt{-3})$.  In other 
words, if $k$ contains $\sqrt{-3}$, $C_{t}$ has a $k(t)$-rational point and 
it is an elliptic curve over $k(t)$.  However, if $k$ does not contain 
$\sqrt{-3}$, we do not know if $C_{t}$ has a $k(t)$-rational point, and whether
we can consider it as an elliptic curve.  
Instead, we need to consider its Jacobian $J_{t}$.

Using an algorithm for calculating an equation of the Jacobian of the
curve given by a quartic equation (see Connell\cite{connell:quartic}),
we see that $J_{t}$ is given by the equation
\begin{multline*}
J_{t} \ : \ 
Y^{2} = X^{3} + \bigl( A t^{8} + 18 B t^{6} -18 A^{2} t^{4} - 54 A B t^{2}
- 27 ( A^{3} + 9 B^{2} ) \bigr) X \\
+ \bigl( B t^{12} - 4 A^{2} t^{10} - 45 A B t^{8} - 270 B^{2} t^{6} 
+ 135 A^{2} B t^{4} \\
- 54 A ( 2 A^{3} + 9 B^{2} ) t^{2} 
- 243 B ( A^{3} + 6 B^{2} ) \bigr).
\end{multline*}

\begin{prop}\label{prop:jacobian}
The elliptic surface associated with the curve $J_{t}$
has eight singular fibers of type \textup{I$_{3}$}
located at $t$ satisfying 
\begin{equation}\label{bad-locus}
t^{8} + 18 A t^{4} + 108 B t^{2} - 27 A^{2} = 0.
\end{equation}
The Mordell-Weil group $J_{t}(\kbar(t))$ contains a
point of infinite order $\gamma_{1}$ given by
\[
\gamma_{1} = \Big(-\frac{1}{27} t^{6} + 5 A t^{2} - 9 B,
\frac{\sqrt{-3}}{243} t ( t^{8} + 162 A t^{4} - 2916 B t^{2} -2187 A^{2})
\Big).
\]
\end{prop}

\begin{proof}
It is easy to determine the singular fibers using Tate's algorithm. 
Over $k(\sqrt{-3})$, $C_{t}$ and $J_{t}$ are isomorphic.
Using an algorithm in \cite{connell:quartic}, we can write an 
isomorphism which send one of the two points at infinity on $C_{t}$ to 
the origin of $J_{t}$ and the other to $\gamma_{1}$.
Using an algorithm in \cite{kuwata:can_height}, we calculate the canonical
height of $\gamma_{1}$, which turns out to be~$3$.  This implies that
it has infinite order.
\end{proof}

\begin{rem}
We note that, if $E$ does not have complex multiplication, then $J_{t}(\kbar(t))$ is
isomorphic to
\[
\Z\oplus\Z/3\Z\oplus\Z/3\Z,
\]
and $J_{t}(\kbar(t))/J_{t}(\kbar(t))_{\text{tors}}$ is generated by
$\gamma_{1}$. All the points in $J_{t}(\kbar(t))$ are defined already
over $k(E[3])(t)$.  We omit the proof since we do not need these 
facts in the sequel.
\end{rem}

\section{Vanishing of Cubic Twists}

Let $E$ be an elliptic curve defined over $\Q$.
The goal of this section is to prove Theorem~\ref{thm:A}.
To do so, we will prove its algebro-geometric version
Theorem~\ref{Theorem1a} and then apply Kato's theorem
to conclude the vanishing of the corresponding twisted
$L$-functions.

\begin{thm}\label{Theorem1a}
Let $E$ be an elliptic curve defined over a number field $k$.
Suppose that there is at least one cyclic extension $K_0/k$ of
degree dividing $3$ such that $E(K_0)$ is infinite (possibly
$K_0=k$). Then $\rk E(K_{\lambda})> \rk E(k)$ for infinitely
many cyclic cubic extensions $K_{\lambda}/k$.
\end{thm}

Let $K/k$ be any finite extension and let $\Tr_{K/k}: E(K) \to E(k)$
denote the trace map.  The kernel $\Ker \Tr_{K/k}\subset E(K)$  is
the subgroup of points of $E(K)$ of trace zero. 
The point satisfying the second condition of
Lemma~\ref{prop:rat_pts_on_SE} belongs to $\Ker \Tr_{K/k}$.

\begin{lem} 
The following are equivalent:
\begin{enumerate}
\item $\rk E(K)>\rk E(k)$,
\item $E(K)$ contains a trace zero point of infinite order,
\item $\#(\Ker \Tr_{K/k})=\infty$.
\end{enumerate}
\end{lem}

\begin{proof}
Let $n=[K:k]$.  Define the maps $t$ and $t'$ of abelian groups as follows.
\[
\begin{array}{ccccc}
E(K) & \overset{t}{\longrightarrow}  & E(k)\times \Ker \Tr_{K/k}
& \overset{t'}{\longrightarrow}  & E(K)
\\[\medskipamount]
P & \longmapsto &
\bigl(\Tr_{K/k}(P), nP -\Tr_{K/k}(P)\bigr) & 
\\
& & (Q, R)& \longmapsto & Q+R
\end{array}
\]
Then we see that $t'\circ t=[n]$ and $t\circ t'=[n]$, where $[n]$ is the multiplication-by-$n$ map. 
Thus, we have
\[ 
 \rk E(K)=\rk E(k)+\rk \Ker \Tr_{K/k}.
\]
The statement of Lemma follows immediately.
\end{proof}

We now prove some lemmas that are necessary later in the proof.

Consider the surface defined by \eqref{quartic} together with the
fibration $(t,u,\delta)\mapsto t$.  Suppose we have infinitely many
$k$-rational points $\gamma_{n}=(u_{n},t_{0},\delta_{n})$ for a fixed
$t_{0}$. 
For each $n$, the point $\gamma_{n}$
corresponds to a class $[P_{n},Q_{n},R_{n}]$ in $S_{E}$.  Let $K_{n}$ be the field
over which $P_{n}$, $Q_{n}$ and $R_{n}$ are defined.  We already know that 
$K_{n}=k$ or $K_{n}/k$ is a cyclic cubic extension of $k$.  

\begin{lem}\label{lem:rank_increase}
Suppose there is a nonzero $t_{0}$ such that the fiber
$C_{t_{0}}=\pi^{-1}(t_{0})$ is a good fiber having infinitely many 
$k$-rational points $\gamma_{n}=(u_{n},t_{0},\delta_{n})$.  Then the 
compositum of all $K_{n}$ is an infinite extension of~$k$.
\end{lem}

\begin{proof}
For each $n$, the cubic polynomial $x^{3} + A x + B - (t_{0} x + u_{n})^{2}$
in $x$ factors into three linear terms over $K_{n}$.
Conversely, finding a $k$-rational point $(u,t_{0},\delta)$ on the 
surface  \eqref{quartic}  involves finding $u$ in $k$ such that 
$x^{3} + A x + B - (t_{0} x + u)^{2}$ factors completely over some cubic 
cyclic field~$L$.  This is equivalent to finding a point 
$(\xi_{1},\xi_{2},\xi_{3},t_{0})$ on the curve given by
\[
\left\{
\begin{aligned}\null
&\xi_{1}+\xi_{2}+\xi_{3} = t_{0}^{2}, \\
&\xi_{1}\xi_{2}+\xi_{2}\xi_{3} + \xi_{3}\xi_{1}= A - 2t_{0}u, \\
&\xi_{1}\xi_{2}\xi_{3} = u^{2} - B.
\end{aligned}
\right.
\]
Since $t_{0}\neq0$, we  obtain a plane curve of degree~$4$ by eliminating $\xi_{3}$ and $u$.  A calculation shows that this degree~$4$ curve is 
nonsingular if and only if  $t_{0}^{8} + 18 A t_{0}^{4} + 108 B t_{0}^{2} - 27 A^{2} \neq 0$.  If that is the 
case, the genus of the curve is~$3$.  Thus, by a theorem of 
Faltings, it has only finitely many $K$-rational points for each fixed 
number field~$K$.  Therefore, the compositum of all $K_{n}$ cannot be
a number field of finite degree over $\Q.$
\end{proof}

In the case where $k$ contains $\sqrt{-3}$, Lemma~\ref{lem:rank_increase} and
Proposition~\ref{prop:jacobian} proves the following statement, which is 
stronger than Theorem~\ref{Theorem1a}.

\begin{thm}\label{thm:with_sqrt(-3)}
Let $E$ be an elliptic curve defined over a number field $k$ 
containing $\sqrt{-3}$.  Then there exist infinitely many cyclic cubic 
extensions $K_{\lambda}$ such that the rank of the Mordell-Weil group 
$\rk E(K_{\lambda}) > \rk E(k)$.\qed
\end{thm}

For the general case we need another lemma.

\begin{lem}\label{lem:zariski_dense}
Let $S$ be a smooth surface and $C$ a smooth curve both defined
over $k$.  Let $\pi:S\to C$ be a fibration defined over $k$ such
that the generic fiber is a curve of genus~$1$ equipped with an
involution $\iota$ with a fixed point. Suppose that the set of
$k$-rational points, $S(k)$, is Zariski dense in $S$, then there
exist infinitely many $k$-rational points $P$ on $C$ such that
the fiber $\pi^{-1}(P)$ contains infinitely many $k$-rational
points.
\end{lem}

\begin{proof}
Let $\pi':J\to C$ be the Jacobian fibration associated with
$\pi:S\to C$.  There is a map $f:S\to J$ of degree~$4$ defined
over~$k$ sending a point $P\in S$ to the divisor class $(P) -
(\iota(P))$.  Since $f$ is dominant, the image of $S(k)$ by $f$
is Zariski dense.
	
By Merel's theorem on the bound for the torsion points defined
over a number field on an elliptic curve (\cite{merel}), the set
consisting of all the $k$-rational torsion points of all the
fibers is contained in a proper Zariski closed set.  Thus if we
denote by $f(S(k))'$ the set consisting of all the points in the
image of $f(S(k))$ that have infinite order, then $f(S(k))'$ is
still Zariski dense in $J$.  This means that there are
infinitely many $k$-rational points $P$ on $C$ such that the
fiber $\pi'{}^{-1}(P)$ contains points in $f(S(k))'$.  For such
$P$ the $\pi^{-1}(P)$ contains infinitely many $k$-rational
points.
\end{proof}

\begin{proof}[Proof of Theorem~\ref{Theorem1a}]
Let $K_{0}/k$ be a cyclic extension of degree dividing $3$ such
that $E(K_{0})$ is infinite. Let $P\in E(K_{0})$ be a point of
infinite order. First, we show that the set of $k$-rational
points in $S_{E}$ is Zariski dense in $S_{E}$.

If $K_{0}=k$ and $P$ is defined already over~$k$, then 
the set $\{[mP,nP,-(m+n)P]\mid n,m\in\Z\}$  is clearly a Zariski
dense in $S_{E}$. We thus assume that $K_{0}/k$ is a
cubic extension and $P$ is defined over~$K_{0}$, but not
over~$k$. Let $\sigma$ be a generator of $\Gal(K_{0}/k)$. Then
$R=\Tr_{K_{0}/k}(P)$ is a point defined over $k$. If $R$ is a
point of infinite order, then we are in the previous case. If
not, replacing $P$ by $nP$ if necessary, we may assume that
$\Tr_{K_{0}/k}(P)=O$.

We consider $E(K_{0})$ as an $\End_{k}(E)$-module, and we claim
that $P$ and $P^{\sigma}$ are $\End_{k}(E)$-linearly
independent, except when $E$ has complex multiplication over
$\Q(\sqrt{-3})$ and $k$ contains $\sqrt{-3}$. Suppose $[\alpha]$
and $[\beta]$ are two nonzero endomorphsims of $E$ defined over~$k$,
and suppose we have the relation
\begin{equation}\label{depend}
[\alpha] P + [\beta] P^{\sigma} = O.
\end{equation}
Apply $\sigma$ to both sides of~\eqref{depend}.  Since $\sigma$
commutes with $[\alpha]$ and $[\beta]$, we have another relation
\begin{equation}\label{depend2}
[\alpha] P^{\sigma} + [\beta] (-P -P^{\sigma}) = O.
\end{equation}
Eliminating $P^{\sigma}$ from \eqref{depend} and \eqref{depend2}, we
obtain
\[
\bigl([\alpha]^{2}-[\alpha][\beta]+[\beta]^{2}\bigr)P=O.
\]
This occurs only when $E$ has complex multiplication by 
$\Q(\sqrt{-3})$. Moreover, since $[\alpha]$ is defined over $k$, 
$\sqrt{-3}$ must be contained in $k$.  We thus verified the claim.
The case where $k$ contains $\sqrt{-3}$ has been treated already.  In 
what follows we assume $\sqrt{-3}\not\in k$.

Next we claim that the subgroup $\{(nP,nP^{\sigma},nP^{\sigma^{2}})\mid n\in\Z\}$  is Zariski dense in $\widehat D$.  
To prove this, it suffices to show that
$\{(nP,nP^{\sigma})\mid n\in\Z\}$ is
Zariski dense in $E\times E$. Let $F$ be the Zariski closure of this
subgroup. Suppose $F$ does not equal $E\times E$, then $F$ is a closed
subgroup of dimension~$1$ in $E\times E$. Let $F^{0}$ be the connected
component of $F$ containing the identity. We then have two isogenies
$\phi_{1}$ and $\phi_{2}$ from $F^{0}$ to $E$, corresponding to two
projections $E\times E\to E$. Choose $m\in\Z$ such that $(m
P,mP^{\sigma})$ is in $F^{0}$. Let $\hat \phi_{1}$ be the dual isogeny
of $\phi_{1}$. Consider the endomorphism $\phi_{2}\hat\phi_{1}$ of
$E$. Let $d$ be the degree of $\phi_{1}$.  Since $\hat\phi_{1}\phi_{1}$ 
equals the multiplication-by-$d$ map, we have
\begin{alignat*}2
\phi_{2}\hat\phi_{1}(mP) 
&= \phi_{2}\hat\phi_{1}\phi_{1}\bigl((mP,mP^{\sigma})\bigr) &&\\
&=\phi_{2}\bigl((dmP,dmP^{\sigma})\bigr) &&\qquad(d = \deg\phi_{1})\\
&=dmP^{\sigma}. &&
\end{alignat*}
This contradicts the independence of $P$ and $P^{\sigma}$.

Since the projection map $\widehat D\to \overline{S}_{E}$ is a
dominant map, the set $\{[nP,nP^{\sigma},nP^{\sigma^{2}}]\mid n\in\Z\}$ is also
Zariski dense in $S_{E}$.  We have thus proved that $S_{E}(k)$ is
Zariski dense in all cases.

The fibration $\pi:\overline{S}_{E}\to \P^{1}$ constructed in
\S4 satisfies the hypotheses of Lemma~\ref{lem:zariski_dense}.
 Thus, there exist infinitely many $t\in\P^{1}$ such that the
fiber $\pi^{-1}(t)$ has infinitely many $k$-rational points.  In
particular, we have at least one such $t$ such that $t\neq0$ and
$\pi^{-1}(t)$ is a good fiber. Then
Lemma~\ref{lem:rank_increase} implies that there exist
infinitely many different cyclic cubic extension $K_{\lambda}$
such that the elliptic curve $E$ possesses a point $P_{\lambda}$
defined over $K_{\lambda}$.

In order to complete the proof we have to show that $P_{\lambda}$ has
infinite order except for finite number of $\lambda$.  But this is
true because the bound of the order of torsion points given by
Merel's theorem depends only on the degree of the field.
\end{proof}

\section{Elliptic curves with at least $6$ rational torsion points}

In this section we prove the following statement.

\begin{thm}\label{thm-torsion}
Let $E$ be an elliptic curve defined over $\Q$. Suppose that
there are at least $6$ points in $E(\Q)$.  Then,  for infinitely
many cyclic cubic extensions $K/\Q$,  we have $\rk E(K )> \rk
E(\Q)$.
\end{thm}

\begin{proof}
If $E(\Q)$ has infinitely many points, then this is nothing but
Theorem~\ref{Theorem1a}. Suppose $E(\Q)$ is finite.  Then, by
Mazur's bound for the torsion of elliptic curve over $\Q$,
either $E(\Q)$ is a cyclic group of order~$6,7,8,9,10, 12$, or
it contains  $\Z/4\Z\times \Z/2\Z$ as a subgroup.  In view of
Lemma~\ref{lem:rank_increase}, it suffices to show that one of
the fibers  $C_{t_{0}}$ of the fibration defined by
\eqref{quartic} has infinitely many rational points.  In the
sequel, we work out in detail to give a particular fiber that
has infinitely many rational points for the case where $E(\Q)$
has a point of order $6$, or $E(\Q)\supset \Z/4\Z\times \Z/2\Z$.  In
the case of higher torsion points, we can prove it similarly to the 
$6$-torsion case. The actual calculations, however, become 
more complicated, and so we omit it here.
\end{proof}

\subsection{Elliptic curve with $6$-torsion point} 
Let us consider the universal elliptic curve having a point of 
order~$6$.  It is given by the equation
\[
y^2 + (1-\lambda) x y - \lambda(\lambda+1) y = x^3 - \lambda(\lambda+1) x^2. 
\]
When $\lambda\neq0$, $-1$ or $-1/9$, this is an elliptic curve
and the point $P=(0,0)$ is a point of order~$6$.  The line
passing through $P$, $2P$ and $3P$ is given by $y=\lambda x$.  
Consider the surface in $\A^{2}=\{(t, u,\delta)\}$ defined by 
\eqref{S_E_affine} with the
fibration $\pi:(t, u,\delta)\mapsto t$ and with parameter
$\lambda$. 

We show that the fiber at $t=\lambda$ has infinitely many
rational points. First, we see that it has two points
$(u,\delta)=(0,\pm \lambda ^{4}(\lambda +1))$ corresponding to
the line $y=\lambda x$ we mentioned above. Choosing one of them,
say $(0,-\lambda ^{4}(\lambda +1))$, as the origin, we can
convert the equation of the fiber $C_{\lambda }=\pi^{-1}(\lambda
)$ into Weierstrass form using the method described in
Connell \cite{connell:quartic}:
\begin{multline}\label{ex1}
C_{\lambda }: y^{2} + (8 \lambda  + 2 \lambda ^{2} + 2) x y - 4 \lambda  (7 \lambda  + 1) (\lambda  - 2) (\lambda  + 1)^{2} y  \\
= x^{3} - 2 \lambda  (\lambda  + 1) (2 \lambda ^{2} - 4 - \lambda ) x^{2} 
 + 108 \lambda ^{4} (\lambda  + 1)^{2} x \\ - 216 \lambda ^{5} (2 \lambda ^{2} - 4 - \lambda ) (\lambda  + 1)^{3}.
\end{multline}
$C_{\lambda }$ is an elliptic curve if and only if $\lambda$ satisfies
\[
\lambda (1 + 9 \lambda )(2 \lambda  + 1)(\lambda  + 1)(\lambda ^4 + 3 \lambda ^3 + 4 \lambda ^2 + 1)\neq0.
\]
The point $(0,\lambda ^{4}(\lambda +1))$ is sent to the point 
$\gamma_{1}=\bigl(2\lambda (\lambda +1)(2\lambda ^2-4-\lambda ), 0\bigr)$.  

\begin{lem}
For all $\lambda \in\Q$ such that $C_{\lambda }$ is an elliptic
curve, the point $\bigl(2\lambda (\lambda +1)(2\lambda
^2-4-\lambda ), 0\bigr)$ is a point of infinite order. When
$\lambda=-1/2$, then $C_{\lambda }$ is not an elliptic curve,
but $(3/2, 0)$ is still a point of infinite order.
\end{lem}

\begin{proof}
We consider $C_{\lambda }$ as the curve defined over $\Q(\lambda
)$, and calculate $n\gamma_{1}$, $n=1,2,..10,12$.  For all those
$n$ we observe that the denominator of the $x$-coordinate of
$n\gamma_{1}$ does not vanish for any value of $\lambda$ except
for $\lambda=0$.  For $\lambda=-1/2$, the group is isomorphic
to $\Q^{\times}$.  Thus, it suffices to see that it is not a
$2$-torsion point.
\end{proof}

\subsection{Elliptic curves with $\Z/4\Z\oplus \Z/2\Z$ torsion}

The elliptic curve
\[
y^{2} = x(x + \mu ^{2})(x + \lambda^{2})
\]
is the universal elliptic curve with $\Z/4\Z\oplus \Z/2\Z$  torsion (\cite{Knapp}). 
Without loss of generality we may set $\mu = 1$ giving
\[
E_{\lambda} : y^{2} = x(x + 1)(x + \lambda^{2})
\]
with $4$-torsion points 
\[
\pm P = (\lambda,\pm \lambda(1 + \lambda)), \quad \pm P'=(-\lambda,\pm \lambda(1-\lambda))
\]
and 2-torsion points 
\[
[2]P= (0, 0),\quad [2]P' = (-1, 0), \quad Q = (-\lambda^{2}, 0).
\]
$E_{\lambda}$ is an elliptic curve for all $\lambda$ different from
$\lambda=0, \pm 1$.

Consider the surface in $\A^{2}=\{(t, u,\delta)\}$ defined by \eqref{S_E_affine}
with the fibration $\pi:(t, u,\delta)\mapsto t$ and with parameter
$\lambda$.
For this case we show that the fiber $t=1$ has infinitely
many points. Setting $t = 1$ and $u = \lambda ^{2}$,  the line
$y = x+ \lambda ^{2}$ passes through three torsion points 
\[
P =
(\lambda , \lambda (1 + \lambda )), \quad Q = (-\lambda ^{2}, 0),
\quad 
-P - Q = (-\lambda ,-\lambda (1 - \lambda )).
\]
This triple of points gives rise to
rational points $(\lambda ^{2},\pm 2\lambda ^{3}(\lambda
^{2}-1))$ in $C_{1}$, and we may proceed to convert the curve to
Weierstrass form using the method described in
Connell \cite{connell:quartic}.  We then simplify it to obtain
\begin{multline*}
C_{1}: y^{2}=x^{3} + 27\lambda ^2 (7 \lambda ^4 - \lambda ^6 
+ 5 \lambda ^2 - 27) x 
\\
+ 27 \lambda ^2 (2 \lambda ^{10} - 21 \lambda ^8 
+ 204 \lambda ^6 - 826 \lambda ^4 +1242 \lambda ^2 - 729),
\end{multline*}
and a rational point 
\[
\bigl(3 (\lambda ^{4} + 16 \lambda ^{2} + 3), \pm  27(7\lambda ^{4} + 10 \lambda ^{2} - 1) \bigr).
\]
The discriminant of $C_{1}$ is
\[
-2^{4}\, 3^{12}\, \lambda ^4 (\lambda -1)^2 (\lambda +1)^2
(3 \lambda ^4 - 14 \lambda ^2 + 27)^3.
\]
\begin{lem}
For all $\lambda \in\Q$ different from $0, \pm 1$, the curve
$C_{1}$ is an elliptic curve and point $\bigl(3 (\lambda ^{4} +
16 \lambda ^{2} + 3), \pm  27(7\lambda ^{4} + 10 \lambda ^{2} -
1) \bigr)$ is a point of infinite order.
\end{lem}

\begin{proof}
Considering the fiber $C_{1}$ as a curve over $\Q(\lambda )$,
this point is a point of infinite order, which can be verified
by a height calculation.  Just as the case of $6$-torsion points,
$C_{1}$ has  infinitely many rational points for all $\lambda
\neq 0$.  For $\lambda=\pm 1$, the group is isomorphic to
$\Q^{\times}$.  Thus, it suffices to see that it is not a
$2$-torsion point.
\end{proof}

\section{An Example - The Curve $E\,37B$.}

We now consider one of the curves of conductor $37$ which is denoted $37B$ in 
Antwerp Table \cite{Antwerp}, and which has Weierstrass equation
\[
E\,37B: y^2+y= x^3+x^2-3x+1.
\]
(In Cremona's tables \cite{Cr}, it is denoted 37b3.)
We decided to study this example because the computations of \cite{F}
indicated an unusually large number of twists by cubic characters 
$\chi$ for which $L(E\,37B,1,\chi)=0$.

Substituting $(x,y)$ by $(x+1,y+2x)$ in the equation above, we obtain another
model of $E\,37B$:
\[
E: y^{2} + 4xy +y = x^{3}.
\]
Note that $(0,0)$ is a point of order~$3$.  Intersecting the line $y=tx+u$ with 
this curve $E$, we obtain an affine model of $\overline{S}_{E}:\delta^{2}=\Delta(u,t)$
with the fibration $(t,u,\delta)\mapsto t$.  The fiber at $t=0$ is a singular fiber 
given by the equation
\[
\delta^{2}=-u^2(27u^2 - 202u + 27).
\]  
This curve has a $\Q$-rational parametrization if and only if
$-(27u^2 - 202u + 27)$ is a square 
for some rational value of $u$.   It turns out that when $u=7/9$, it becomes
$(32/3)^{2}$.  Using this solution, we can parametrize the fiber at $t=0$:
\[
u=\frac{7r^{2}+12r+9}{9r^{2}-12r+7},
\qquad
\delta=\frac{32(7r^{2}+12r+9)(3r^{2}+r-3)}
{(9r^{2}-12r+7)^{2}},
\]
where $r$ is the parameter.  This means that the cubic equation $u^{2}+4xu+u=x^{3}$ 
in $x$ with $u$ given by the above formula is a cyclic polynomial.  Let $\xi_{r}$ be 
a root of the cubic polynomial
\[
F_{r}(Z)=Z^3-4(7r^{2}+12r+9)(9r^{2}-12r+7)Z
- 16(r^{2}+1)(7r^{2}+12r+9)(9r^{2}-12r+7). 
\]
Then, $K_{r}=\Q(\xi_{r},r)$ is a cyclic cubic extension of the
field $\Q(r)$, and the point 
\[
P_{r} = \left(\frac{\xi_{r}}{9r^{2}-12r+7},\frac{7r^{2}+12r+9}{9r^{2}-12r+7}\right)
\]
belongs to $E(K_{r})$.  A straightforward height calculation shows that $P_{r}$ 
is a point of infinite order.  

Let $a$, $b$ be relatively prime integers.  By Silverman's result (\cite{Si}) the 
specializations of $P_{r}$ to $r=a/b$ have also infinite order except maybe 
for finitely many exceptions.

The discriminant of $K_{r}/\Q(r)$ is 
\[
2^{10}(7r^2+12r+9)^2(9r^2-12r+7)^2(3r^2+r-3)^2.
\]
Let $K_{a/b}$ be the specialization of $K_{r}$ with $r=a/b$.  Then, $K_{a/b}$ 
has square discriminant dividing
\[
2^{10}(7a^{2}+12a+9b^{2})^2(9a^2-12ab+7b^{2})^2(3a^{2}+ab-3b^{2})^2.
\]
and the conductor of $K_{a/b}$ is the square root of its discriminant.
We let
\[
H_{1}=7a^{2}+12a+9b^{2}, \quad H_{2}=9a^2-12ab+7b^{2}, 
\quad
G=a^{2}+b^{2}.
\]
We note that the resultants of any pair of $H_{1}$, $H_{2}$ 
and $G$ is supported only at the primes $2,3$ and $37$. 
Hence, if $p$ is a rational prime  $p \neq 2,3$ or $37$, which divides
$H_{1}H_{2}$ to the first power, then
$b^6F_{a/b}(Z)$ is an Eisenstein polynomial at $p$, and therefore
$p$ ramifies (totally) in $K_{a/b}/\Q$. 

On the other hand, if $p$
divides $3a^{2}+ab-3b^{2}$, then we have
\[
H_{1}H_{2}- 27G^{2} 
=4(3a^{2}+ab-3b^{2})\equiv 0
\bmod p,
\]
and thus
\[
b^{6}F_{a/b}(Z)\equiv Z^3-108\,G^2 Z-432\,G^3
\equiv \bigl(Z+6G\bigr)^{2}\bigl(Z-12G\bigr)
\bmod p.
\]
It
follows that the completion of $K_{a/b}$ at the prime over $p$ which
contains $Z-12G$ is $\Q_p$ and hence $p$ splits completely in
$K_{a/b}$.

Therefore we see that the conductor of $K_{a/b}$ divides
$2^{10}H_{1}H_{2}$. But $H_{1}H_{2}$ is a separable 
binary form of degree $4$  which is primitive. It follows from Stewart-Top
(\cite[Theorem 1]{St-T}) that the number of square free values
less than $X$ of $H_{1}H_{2}$ is $\gg X^{1/2}$. We note that
distinct square free values of $H_{1}H_{2}$ yield distinct fields
$K_{a/b}$.

Therefore we have proved:

\begin{thm} For the elliptic curve $E\,37B$, the number of cubic 
Dirichlet characters $\chi$ for which 
$L(E37B,1,\chi)= 0$ satisfies
 \[
\N_E(3,X)= \#\{\chi \in {\X}(3)\mid \f < X,  L(E37B,1,\chi)= 0\} \gg X^{1/2}.
 \]
\end{thm}

We note that the calculations done above for the curve $E37B$
actually work for any elliptic curve $E/\Q$ with a $\Q$-rational point $P$
of order $3$  and which satisfies the condition below.
We send the point $P$ to $(0,0)$ and express $E$ in the form
\[
y^{2}+3Uxy+Ty=x^{3}.
\]
where $U,T\in\Q.$ The fiber over  $t=0$, on
the surface $\overline{S}_E$ takes the  form
\[
\delta^{2}=-27u^{2}\bigl(u^{2}-(4U^{3}-2T)u+T^{2}\bigr).
\]
This may be expressed as a conic in the variables $z=\delta/3u$ and 
$w=u+2U^3-T$
\[
z^{2}+3w^{2}=12U^3(U^3-T).
\]
This is a  curve of genus zero and may be parametrized over $\Q$
if a single rational point can be found. This occurs if and only
if the right hand side is a norm from  $\Q(\sqrt{-3})$, i.e., if
and only if $U(U^3-T)$ is a norm from  $\Q(\sqrt{-3})$.

These curves give examples for which $\N_E(3,X)\gg X^{1/2}$ mentioned
in the introduction.

\end{document}